\documentclass[11pt]{article}
\usepackage[a4paper,margin=2.7cm]{geometry}
\usepackage{amsmath,amssymb,amsthm,mathtools}
\usepackage{graphicx}
\usepackage{enumitem}
\usepackage{microtype}
\usepackage{booktabs}
\usepackage{float}
\usepackage[pagebackref]{hyperref}
\usepackage{authblk}
\newtheorem{theorem}{Theorem}[section]

\newtheorem{lemma}[theorem]{Lemma}
\newtheorem{corollary}[theorem]{Corollary}
\newtheorem{remark}[theorem]{Remark}
\theoremstyle{definition}
\newtheorem{definition}[theorem]{Definition}

\newcommand{\Om}{\Omega}
\newcommand{\GD}{\Gamma_D}
\newcommand{\GN}{\Gamma_N}
\newcommand{\GR}{\Gamma_R}

\newcommand{\Tr}{\operatorname{Tr}}
\newcommand{\QT}{Q_T}

\title{Weak solvability of a nonlinear Boussinesq groundwater-flow model with mixed boundary conditions: the Moche--CHAVIMOCHIC aquifer}
\author[1]{Alexis Rodriguez Carranza}
\author[2]{V\'ictor Arturo Mart\'inez Le\'on}
\author[3]{Alan Ch\'avez Obreg\'on}
\affil[1]{Instituto de Investigaci\'on en Matem\'aticas, Departamento de Matem\'aticas, FCFYM, Universidad Nacional de Trujillo, Trujillo, Peru.\\
\texttt{arodriguezca@unitru.edu.pe}; ORCID: \href{https://orcid.org/0000-0002-0290-165X}{0000-0002-0290-165X}}
\affil[2]{Universidade Federal da Integra\c{c}\~ao Latino-Americana (UNILA), Instituto Latino-Americano de Ci\^encias da Vida e da Natureza (ILACVN), Foz do Igua\c{c}u, Paran\'a, Brazil.\\
\texttt{victor.leon@unila.edu.br}; ORCID: \href{https://orcid.org/0000-0002-2082-6665}{0000-0002-2082-6665}}
\affil[3]{OASIS and GRACOCC research groups, Instituto de Investigaci\'on en Matem\'aticas, Departamento de Matem\'aticas, FCFYM, Universidad Nacional de Trujillo, Trujillo, Peru.\\
\texttt{ajchavez@unitru.edu.pe}; ORCID: \href{https://orcid.org/0000-0001-5120-0705}{0000-0001-5120-0705}}
\date{}

\begin{document}
\maketitle

\begin{abstract}
We prove weak solvability for a two-dimensional nonlinear Boussinesq-type groundwater-flow problem motivated by the Moche--CHAVIMOCHIC aquifer in northern Peru. The model couples state-dependent storage and transmissivity with a boundary decomposition into Dirichlet, Neumann, and nonlinear Robin parts dictated by the hydrogeological geometry. In contrast with much of the groundwater Boussinesq literature, which is devoted to exact, similarity, perturbative, semi-analytical, or linearized solutions for special geometries and boundary data, our objective is a variational existence result for a genuinely two-dimensional mixed-boundary problem. Because the physical transmissivity is proportional to the saturated thickness, the equation degenerates at local depletion. We therefore isolate a physically admissible saturated range and construct uniformly positive extensions of the constitutive laws. For the resulting uniformly parabolic problem, existence of a weak solution is obtained through an implicit Euler--Rothe scheme, a nonlinear energy adapted to the storage law, uniform estimates for the discrete derivative of the storage potential, discrete compactness, and strong convergence of boundary traces. We then give a precise consistency statement showing when the abstract solution solves the original Moche model and identify the fully depleted regime as a distinct porous-medium-type problem.
\end{abstract}


\section{Introduction}
Groundwater flow in unconfined aquifers is a classical source of nonlinear diffusion problems. Under Dupuit-type assumptions, the water-table dynamics are described by Boussinesq-type equations in which the transmissivity depends on the saturated thickness and may therefore vanish at local depletion \cite{Boussinesq,Bear,FreezeCherry}. This structure is analytically delicate even before realistic boundary interactions are introduced. In field-scale models, however, different pieces of the aquifer boundary generally represent different hydraulic mechanisms, so that one is naturally led to mixed boundary conditions rather than to a single homogeneous boundary law.

The Moche--CHAVIMOCHIC aquifer in the La Libertad region of Peru provides a concrete example of this situation. The lower Moche valley is predominantly alluvial and unconfined, and the hydrogeological description distinguishes an ocean interface, essentially impermeable lateral sectors, and a principal recharge connection with the middle basin; see \cite{INRENA}. These three mechanisms lead, respectively, to Dirichlet, Neumann, and Robin conditions. The resulting model is therefore not merely a standard Boussinesq equation posed on a rectangular test geometry: it is a two-dimensional nonlinear parabolic problem whose boundary decomposition is inherited from the physical aquifer.

The groundwater Boussinesq equation has generated a large literature, but a substantial part of it pursues explicit or reduced-form solutions under special geometries and boundary data. For example, Moutsopoulos studied the nonlinear Boussinesq equation with a nonlinear Robin condition describing a semipervious interface and derived a semi-analytical solution \cite{Moutsopoulos2013}. Exact solutions for horizontal and sloping aquifers with source and sink terms were obtained by Bartlett and Porporato \cite{BartlettPorporato2018}; semi-analytical and perturbation solutions for one-dimensional horizontal aquifers were developed in \cite{Hayek2019,Basha2021}; and more recent work continues this program through closed-form, approximate, and similarity constructions \cite{Hayek2024,AkylasGravanis2024,TogoUnami2026}. Thus, nonlinear Robin exchange and nonlinear Boussinesq dynamics are individually well represented in the hydrological literature, but typically in one-dimensional or specially reducible configurations and with the main emphasis on analytical or approximate solution formulas rather than on variational weak solvability.

Recent two-dimensional studies likewise emphasize analytical representations after linearization or geometric simplification. Wu and Hsieh obtained an analytical solution for two-dimensional flow in a sloping unconfined aquifer under spatiotemporal recharge using a linearized Boussinesq equation \cite{WuHsieh2024}, while Hsieh, Yu and Wu developed a two-dimensional analytical model in a rectangular anisotropic aquifer, again within a linearized framework \cite{HsiehYuWu2025}. These works are important for groundwater prediction and benchmarking, but their mathematical objective differs from the one pursued here: we retain the nonlinear constitutive dependence and seek an existence theorem in a Sobolev setting for a mixed Dirichlet--Neumann--Robin boundary decomposition.

There is, of course, a broad PDE theory for nonlinear parabolic and elliptic--parabolic equations based on monotone and pseudomonotone operators, beginning with classical frameworks such as Lions, Alt--Luckhaus, and Showalter \cite{Lions,AltLuckhaus,Showalter}. Robin and mixed boundary conditions have also been studied abstractly for parabolic equations; see, for instance, Nittka's weak-solution theory for inhomogeneous Robin problems and the mixed-boundary existence results of Kim and Cao \cite{Nittka2014,KimCao2018}. Our contribution is not a new abstract existence principle. Rather, it is to identify a rigorous variational formulation for the Moche groundwater model, verify the structural hypotheses generated by its storage and transmissivity laws, and combine these with the physically induced three-part boundary decomposition.

More precisely, the physical coefficients are
\[
S_{\rm ph}(x,r)=n_e+S_s(r-\eta(x)),\qquad
T_{\rm ph}(x,r)=K(r-\eta(x)),
\]
so the transmissivity loses ellipticity when $r=\eta(x)$. We separate two regimes. In the present paper we work on a physically admissible saturated interval on which $r-\eta(x)$ stays uniformly positive, extend the coefficients outside that interval to a uniformly parabolic problem, and prove existence of a weak solution for the extended mixed-boundary system. The fully depleted case is not hidden inside the theorem: it is explicitly identified as a genuinely degenerate porous-medium-type problem requiring a different analysis.

The analytical contribution is fourfold. First, we formulate the Dirichlet--Neumann--Robin problem in a Sobolev framework that accommodates a nonlinear monotone exchange law on the Robin portion. Second, we introduce the storage potential
\[
B(x,r)=\int_0^r S(x,s)\,ds
\]
and the associated energy, which are the natural objects for the nonlinear time derivative. Third, we construct solutions by an implicit Euler--Rothe scheme and obtain uniform estimates for the discrete derivative of $B(\cdot,h)$. Fourth, a discrete compactness argument together with strong compactness of traces allows passage to the nonlinear diffusion and Robin terms. The main theorem is therefore an existence theorem; uniqueness is deliberately not asserted without additional monotonicity assumptions.

The novelty should accordingly be understood as the combination of the following features within one rigorous weak-solvability analysis: a genuinely two-dimensional nonlinear groundwater model, state-dependent storage and transmissivity, a physically derived Dirichlet--Neumann--Robin partition of the boundary, and a nonlinear Robin exchange law. The most directly related groundwater studies cited above address important pieces of this picture, but with different goals---chiefly explicit, approximate, similarity, numerical, or linearized solutions. Our result complements that literature by providing the variational existence side of the model.

The paper is organized as follows. Section~\ref{sec:model} presents the Moche groundwater model and the saturated-regime reduction. Section~\ref{sec:functional} introduces the functional framework and the nonlinear Robin law. Section~\ref{sec:existence} proves the existence theorem through time discretization and compactness. Section~\ref{sec:physical} states precisely how the abstract result applies to the physical coefficients. Section~\ref{sec:degenerate} discusses the depleted, degenerate regime and the analytical issues that remain beyond the uniformly saturated setting.

\section{The Moche groundwater model}\label{sec:model}
Let $\Om\subset\mathbb R^2$ be a bounded Lipschitz domain representing the lower Moche aquifer. We assume that
\[
\partial\Om=\overline{\GD}\cup\overline{\GN}\cup\overline{\GR},
\]
where $\GD,\GN,\GR$ are pairwise disjoint relatively open sets and $|\GD|>0$. The sets represent, respectively, the ocean boundary, the essentially impermeable lateral boundary, and the recharge boundary connecting the lower and middle basins. The hydrogeological configuration motivating this decomposition is illustrated in Figure~\ref{fig:moche-boundary}. The Pacific Ocean provides the reference-head boundary, the lateral sectors adjacent to rock formations and desert areas motivate prescribed-flux conditions, and the inland recharge sector associated with the Moche--CHAVIMOCHIC system motivates the exchange condition.

\begin{figure}[H]
\centering
\includegraphics[width=0.86\textwidth]{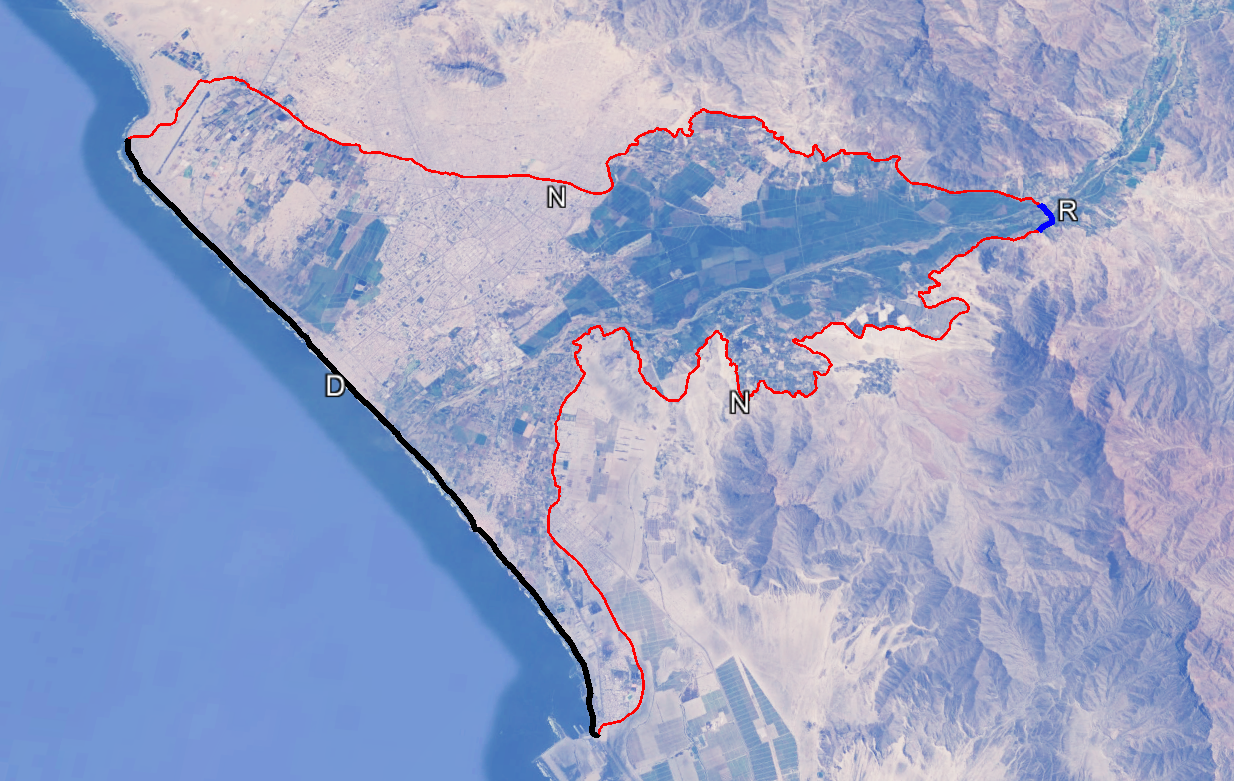}
\caption{Hydrogeological configuration of the lower Moche--CHAVIMOCHIC aquifer motivating the boundary decomposition into ocean-interface ($\Gamma_D$), prescribed-flux ($\Gamma_N$), and recharge/exchange ($\Gamma_R$) portions. The labels D, N, and R indicate these three boundary types.}
\label{fig:moche-boundary}
\end{figure}

Accordingly, the mathematical boundary partition used below is the idealized representation of the three hydraulic mechanisms displayed in Figure~\ref{fig:moche-boundary}; the figure is intended to motivate the model geometry rather than to assert an exact geometric reconstruction of the aquifer boundary.

The physical model is
\begin{equation}\label{eq:physical}
\frac{\partial}{\partial x}\left((h-\eta)h_x\right)
+\frac{\partial}{\partial y}\left((h-\eta)h_y\right)
+\frac{N}{K}
=
\frac{n_e}{K}h_t+\frac{S_s}{K}(h-\eta)h_t,
\end{equation}
where $h=h(x,t)$ is the piezometric head, $K>0$ is the hydraulic conductivity, $n_e>0$ is the effective porosity, $S_s\ge0$ is the specific storage coefficient, $N$ is a source/recharge term, and $\eta=\eta(x)$ is the lower reference level entering the saturated thickness. We assume
\begin{equation}\label{eq:eta}
\eta\in W^{1,\infty}(\Om).
\end{equation}
Equivalently,
\begin{equation}\label{eq:physical2}
\bigl(n_e+S_s(h-\eta)\bigr)h_t
-K\,\operatorname{div}\bigl((h-\eta)\nabla h\bigr)=N.
\end{equation}
Thus the physical constitutive functions are
\begin{equation}\label{eq:physicalST}
S_{\rm ph}(x,r)=n_e+S_s(r-\eta(x)),
\qquad
T_{\rm ph}(x,r)=K(r-\eta(x)).
\end{equation}

The boundary conditions are
\begin{align}
h&=0 &&\text{on }\GD\times(0,T),\label{eq:Dir}\tag{D}\\
-T_{\rm ph}(x,h)\,\partial_n h&=g &&\text{on }\GN\times(0,T),\label{eq:Neu}\tag{N}\\
-T_{\rm ph}(x,h)\,\partial_n h&=\alpha(h)(h-h_{\rm ext}) &&\text{on }\GR\times(0,T),\label{eq:Rob}\tag{R}
\end{align}
and
\begin{equation}\label{eq:init}
h(\cdot,0)=h_0\quad\text{in }\Om.
\end{equation}
The Dirichlet condition corresponds to the choice of sea level as reference head. The Neumann condition models a prescribed flux on lateral sectors, while the Robin law describes hydraulic exchange through the principal recharge zone.

\subsection{The saturated-regime reduction}
The physical transmissivity vanishes when $h=\eta$. Hence \eqref{eq:physical2} is not uniformly parabolic through complete local depletion. We therefore distinguish the saturated regime from the depleted regime.

Fix a physically admissible interval $I=[m,M]$ and assume that
\begin{equation}\label{eq:admissible}
r-\eta(x)\ge\delta>0
\qquad\text{for a.e. }x\in\Om\text{ and every }r\in I.
\end{equation}
On $I$ the physical coefficients satisfy
\[
S_{\rm ph}(x,r)\ge n_e>0,
\qquad
T_{\rm ph}(x,r)\ge K\delta>0.
\]
Choose extensions $S,T:\Om\times\mathbb R\to\mathbb R$ which coincide with $S_{\rm ph},T_{\rm ph}$ on $I$ and satisfy the hypotheses stated below. The existence theorem is proved for the extended uniformly parabolic problem. If the resulting solution takes values in $I$, then it is a weak solution of the physical Moche model itself; see Corollary~\ref{cor:physical}. Since the ocean condition prescribes $h=0$ on $\GD$, a physically invariant interval intended to contain the full solution should in particular be compatible with this boundary value (typically $0\in I$).

\section{Functional setting and weak formulation}\label{sec:functional}
Set
\[
V:=\{v\in H^1(\Om):\Tr v=0\text{ on }\GD\}.
\]
By the Poincar\'e inequality, $\|v\|_V:=\|\nabla v\|_{L^2(\Om)}$ is an equivalent norm on $V$, and
\[
V\hookrightarrow L^2(\Om)\hookrightarrow V'
\]
is a Gelfand triple.

We impose the following structural hypotheses on the extended coefficients.

\medskip
\noindent\textbf{(H1) Storage and transmissivity.}
The functions $S,T$ are Carath\'eodory functions and there exist constants $S_0,S_1,T_0,T_1>0$ such that
\begin{equation}\label{eq:coeff}
0<S_0\le S(x,r)\le S_1,
\qquad
0<T_0\le T(x,r)\le T_1
\end{equation}
for a.e. $x\in\Om$ and all $r\in\mathbb R$.

Define
\begin{equation}\label{eq:B}
B(x,r):=\int_0^r S(x,s)\,ds.
\end{equation}
Then
\begin{equation}\label{eq:BbiLip}
S_0|r-s|\le |B(x,r)-B(x,s)|\le S_1|r-s|.
\end{equation}
We additionally assume the following Sobolev chain-rule property. For every $u\in H^1(\Om)$,
$B(\cdot,u)\in H^1(\Om)$ and
\begin{equation}\label{eq:Bchain}
\nabla[B(x,u(x))]
=S(x,u(x))\nabla u(x)+\nabla_xB(x,u(x))
\quad\text{a.e. in }\Om,
\end{equation}
where, for some $C_B>0$,
\begin{equation}\label{eq:Bx}
|\nabla_x B(x,r)|\le C_B(1+|r|)
\qquad\text{for a.e. }x\in\Om,\ r\in\mathbb R.
\end{equation}
This assumption is, for example, satisfied if $B$ has the usual local Sobolev regularity in $x$, is $C^1$ in $r$, and $\partial_rB=S$ with the bounds above. For the physical storage law in \eqref{eq:physicalST}, it follows directly from \eqref{eq:eta}, since
\[
B_{\rm ph}(x,r)=n_e r+\frac{S_s}{2}r^2-S_s\eta(x)r,
\qquad
\nabla_xB_{\rm ph}(x,r)=-S_s r\nabla\eta(x).
\]
The extensions may be chosen so that \eqref{eq:Bchain}--\eqref{eq:Bx} are preserved.

\medskip
\noindent\textbf{(H2) Nonlinear Robin exchange.}
Let $\Sigma_R:=\GR\times(0,T)$. We write the Robin law abstractly as
\[
-T(x,h)\partial_nh=\beta(x,t,h)\qquad\text{on }\Sigma_R.
\]
The function $\beta:\Sigma_R\times\mathbb R\to\mathbb R$ is Carath\'eodory, nondecreasing in its last variable,
\begin{equation}\label{eq:betamono}
(\beta(x,t,r)-\beta(x,t,s))(r-s)\ge0,
\end{equation}
and has at most linear growth:
\begin{equation}\label{eq:betagrowth}
|\beta(x,t,r)|\le b_0(x,t)+b_1|r|,
\qquad b_0\in L^2(\Sigma_R),\quad b_1\ge0.
\end{equation}
The physical law
\[
\beta(x,t,r)=\alpha(r)(r-h_{\rm ext}(x,t))
\]
is covered whenever it satisfies \eqref{eq:betamono} and \eqref{eq:betagrowth}. In particular, this holds for a positive constant exchange coefficient $\alpha$ with $h_{\rm ext}\in L^2(\Sigma_R)$.

\medskip
\noindent\textbf{(H3) Data.}
We assume
\begin{equation}\label{eq:data}
h_0\in L^2(\Om),\qquad
q\in L^2(0,T;V'),\qquad
g\in L^2(0,T;H^{-1/2}(\GN)).
\end{equation}

\begin{definition}[Weak solution]\label{def:weak}
A function
\[
h\in L^\infty(0,T;L^2(\Om))\cap L^2(0,T;V)
\]
is a weak solution of the extended problem if
\[
B(\cdot,h)\in L^2(0,T;H^1(\Om)),
\qquad
\partial_tB(\cdot,h)\in L^2(0,T;V'),
\]
and, for every $v\in V$ and every $\psi\in C^1([0,T])$ with $\psi(T)=0$,
\begin{align}
-\int_0^T (B(\cdot,h),v)_{L^2}\psi'(t)\,dt
&+\int_0^T\!\int_\Om T(x,h)\nabla h\cdot\nabla v\,\psi\,dxdt\notag\\
&+\int_0^T\!\int_{\GR}\beta(x,t,h)\,v\,\psi\,dSdt\notag\\
&=\int_0^T\bigl(\langle q,v\rangle_{V',V}
-\langle g,\Tr v\rangle_{H^{-1/2},H^{1/2}}\bigr)\psi\,dt\notag\\
&\quad +(B(\cdot,h_0),v)_{L^2}\psi(0).
\label{eq:weakint}
\end{align}
\end{definition}

\section{Existence by implicit Euler--Rothe approximation}\label{sec:existence}
\begin{theorem}[Existence]\label{thm:existence}
Assume \emph{(H1)--(H3)}. Then the extended mixed-boundary problem admits at least one weak solution in the sense of Definition~\ref{def:weak}.
\end{theorem}

\subsection{Time discretization}
Let $N\in\mathbb N$, $\tau=T/N$, and $t_k=k\tau$. Put $I_k=(t_{k-1},t_k]$. Let
\[
q^k:=\frac1\tau\int_{I_k}q(t)\,dt,
\qquad
g^k:=\frac1\tau\int_{I_k}g(t)\,dt,
\]
and define
\[
\beta^k(x,r):=\frac1\tau\int_{I_k}\beta(x,t,r)\,dt.
\]
Then $\beta^k$ is nondecreasing in $r$ and satisfies a growth estimate of the form \eqref{eq:betagrowth}, with the time-average of $b_0$ in place of $b_0$.

Set $h^0=h_0$. For $k=1,\dots,N$ we seek $h^k\in V$ such that
\begin{align}
\left(\frac{B(\cdot,h^k)-B(\cdot,h^{k-1})}{\tau},v\right)_{L^2}
&+\int_\Om T(x,h^k)\nabla h^k\cdot\nabla v\,dx\notag\\
&+\int_{\GR}\beta^k(x,h^k)v\,dS\notag\\
&=\langle q^k,v\rangle-\langle g^k,\Tr v\rangle
\label{eq:discrete}
\end{align}
for every $v\in V$.

\begin{lemma}[Solvability of each elliptic step]\label{lem:step}
For every $k=1,\dots,N$, problem \eqref{eq:discrete} has at least one solution $h^k\in V$.
\end{lemma}

\begin{proof}
For fixed $k$, define $\mathcal A_k:V\to V'$ by
\begin{align*}
\langle\mathcal A_k(u),v\rangle
&:=\frac1\tau(B(\cdot,u),v)_{L^2}
+\int_\Om T(x,u)\nabla u\cdot\nabla v\,dx
+\int_{\GR}\beta^k(x,u)v\,dS.
\end{align*}
We decompose
\[
\mathcal A_k=\frac1\tau\mathcal B+A+R_k,
\]
where
\[
\langle\mathcal B(u),v\rangle=(B(\cdot,u),v)_{L^2},\qquad
\langle A(u),v\rangle=\int_\Om T(x,u)\nabla u\cdot\nabla v\,dx,
\]
and
\[
\langle R_k(u),v\rangle=\int_{\GR}\beta^k(x,u)v\,dS.
\]

We first record the properties of the zero-order term. By \eqref{eq:BbiLip},
\begin{equation}\label{eq:B-Lipschitz-L2}
\|B(\cdot,u)-B(\cdot,v)\|_{L^2(\Om)}
\le S_1\|u-v\|_{L^2(\Om)},
\end{equation}
so $\mathcal B:V\to V'$ is bounded and continuous. Moreover,
\begin{equation}\label{eq:B-strongmono}
\langle\mathcal B(u)-\mathcal B(v),u-v\rangle
=\int_\Om [B(x,u)-B(x,v)](u-v)\,dx
\ge S_0\|u-v\|_{L^2(\Om)}^2.
\end{equation}

Next we verify the principal part. We claim that $A:V\to V'$ is of type $(S_+)$. Let
$u_n\rightharpoonup u$ in $V$ and suppose that
\[
\limsup_{n\to\infty}\langle A(u_n),u_n-u\rangle\le0.
\]
Since $V\Subset L^2(\Om)$, after passage to a subsequence,
$u_n\to u$ strongly in $L^2(\Om)$ and almost everywhere in $\Om$. Hence, by the
Carath\'eodory property and \eqref{eq:coeff},
\[
T(x,u_n)\to T(x,u)\quad\text{a.e. in }\Om,
\qquad |T(x,u_n)|\le T_1.
\]
Using $T\ge T_0$, we obtain
\begin{align*}
T_0\|\nabla(u_n-u)\|_{L^2(\Om)}^2
&\le \int_\Om T(x,u_n)|\nabla(u_n-u)|^2\,dx\\
&=\langle A(u_n),u_n-u\rangle
-\int_\Om T(x,u_n)\nabla u\cdot\nabla(u_n-u)\,dx.
\end{align*}
The last integral tends to zero. Indeed,
\[
T(x,u_n)\nabla u
=[T(x,u_n)-T(x,u)]\nabla u+T(x,u)\nabla u,
\]
where the first term converges strongly to zero in $L^2(\Om)^2$ by dominated
convergence, while the second is fixed in $L^2(\Om)^2$ and
$\nabla(u_n-u)\rightharpoonup0$ in $L^2(\Om)^2$. Consequently,
\[
\limsup_{n\to\infty}T_0\|\nabla(u_n-u)\|_{L^2(\Om)}^2\le0,
\]
and therefore $u_n\to u$ strongly in $V$. Thus $A$ is of type $(S_+)$ and, in
particular, pseudomonotone. Notice also that if $u_n\to u$ strongly in $V$, then
$A(u_n)\to A(u)$ strongly in $V'$; this follows from the same dominated-convergence
argument together with $\nabla u_n\to\nabla u$ in $L^2(\Om)^2$.

For the Robin term, the trace theorem and \eqref{eq:betagrowth} give
\begin{align*}
|\langle R_k(u),v\rangle|
&\le \|\beta^k(\cdot,u)\|_{L^2(\GR)}\|\Tr v\|_{L^2(\GR)}\\
&\le C\bigl(\|b_0^k\|_{L^2(\GR)}+\|u\|_V\bigr)\|v\|_V,
\end{align*}
so $R_k:V\to V'$ is bounded. Furthermore, \eqref{eq:betamono} implies
\begin{equation}\label{eq:R-monotone}
\langle R_k(u)-R_k(v),u-v\rangle
=\int_{\GR}[\beta^k(x,u)-\beta^k(x,v)](u-v)\,dS\ge0.
\end{equation}
The Carath\'eodory property and the linear-growth bound imply continuity of the
associated Nemytskii map on $L^2(\GR)$. Hence $R_k$ is hemicontinuous (indeed,
continuous after composition with the trace map). Thus $R_k$ is monotone and
hemicontinuous, and therefore pseudomonotone. We now verify directly that the full
operator $\mathcal A_k$ is of type $(S_+)$. Let $u_n\rightharpoonup u$ in $V$ and assume
\[
\limsup_{n\to\infty}\langle\mathcal A_k(u_n),u_n-u\rangle\le0.
\]
Since $\mathcal B$ and $R_k$ are monotone,
\[
\langle\mathcal B(u_n)-\mathcal B(u),u_n-u\rangle\ge0,
\qquad
\langle R_k(u_n)-R_k(u),u_n-u\rangle\ge0.
\]
Moreover, $\mathcal B(u),R_k(u)\in V'$ and $u_n-u\rightharpoonup0$ in $V$; hence
\[
\langle\mathcal B(u),u_n-u\rangle\to0,
\qquad
\langle R_k(u),u_n-u\rangle\to0.
\]
It follows that
\[
\limsup_{n\to\infty}\langle A(u_n),u_n-u\rangle\le0.
\]
The $(S_+)$ property of $A$ proved above therefore gives $u_n\to u$ strongly in $V$.
The continuity of $\mathcal B$, $A$, and $R_k$ then yields
$\mathcal A_k(u_n)\to\mathcal A_k(u)$ in $V'$. Thus $\mathcal A_k$ is of type
$(S_+)$ and, in particular, pseudomonotone \cite{Lions,Showalter}.

It remains to verify coercivity. Since $B(x,0)=0$, \eqref{eq:B-strongmono} gives
\[
B(x,r)r\ge S_0r^2.
\]
Moreover, by monotonicity of $\beta^k$,
\[
\beta^k(x,r)r\ge \beta^k(x,0)r.
\]
Therefore, using the trace theorem and Young's inequality,
\begin{align*}
\langle\mathcal A_k(u),u\rangle
&\ge \frac{S_0}{\tau}\|u\|_{L^2(\Om)}^2
+T_0\|u\|_V^2
+\int_{\GR}\beta^k(x,0)u\,dS\\
&\ge \frac{S_0}{2\tau}\|u\|_{L^2(\Om)}^2
+\frac{T_0}{2}\|u\|_V^2-C_k,
\end{align*}
where $C_k$ depends on $\|\beta^k(\cdot,0)\|_{L^2(\GR)}$. Thus $\mathcal A_k$ is
bounded, coercive, and pseudomonotone. The pseudomonotone surjectivity theorem
\cite{Lions,Showalter} yields a solution of
\[
\mathcal A_k(h^k)=\frac1\tau B(\cdot,h^{k-1})+q^k-g^k,
\]
where $g^k$ is understood through the trace duality. This is precisely
\eqref{eq:discrete}.
\end{proof}

\subsection{Discrete energy estimate}
Define
\begin{equation}\label{eq:Phi}
\Phi(x,r):=\int_0^r sS(x,s)\,ds.
\end{equation}
Then
\begin{equation}\label{eq:Phibounds}
\frac{S_0}{2}r^2\le \Phi(x,r)\le \frac{S_1}{2}r^2.
\end{equation}
Moreover, for every $a,b\in\mathbb R$,
\begin{equation}\label{eq:discretechain}
(B(x,a)-B(x,b))a\ge \Phi(x,a)-\Phi(x,b).
\end{equation}
Indeed,
\[
(B(x,a)-B(x,b))a-[\Phi(x,a)-\Phi(x,b)]
=\int_b^a(a-s)S(x,s)\,ds\ge0.
\]

Choose $v=h^k$ in \eqref{eq:discrete}. By \eqref{eq:discretechain} and $T\ge T_0$,
\begin{align}
\frac1\tau\int_\Om\bigl[\Phi(x,h^k)-\Phi(x,h^{k-1})\bigr]dx
+T_0\|h^k\|_V^2
+\int_{\GR}\beta^k(x,h^k)h^k\,dS
\le \langle q^k,h^k\rangle-\langle g^k,\Tr h^k\rangle.
\label{eq:discenergy0}
\end{align}
By monotonicity,
\[
\beta^k(x,h^k)h^k\ge \beta^k(x,0)h^k.
\]
The trace theorem and Young's inequality therefore imply
\[
\left|\int_{\GR}\beta^k(x,0)h^k\,dS\right|
\le \varepsilon\|h^k\|_V^2+C_\varepsilon\|\beta^k(\cdot,0)\|_{L^2(\GR)}^2.
\]
Similarly,
\[
|\langle q^k,h^k\rangle|+|\langle g^k,\Tr h^k\rangle|
\le 2\varepsilon\|h^k\|_V^2
+C_\varepsilon\bigl(\|q^k\|_{V'}^2+\|g^k\|_{H^{-1/2}(\GN)}^2\bigr).
\]
Since $|\beta^k(x,0)|\le b_0^k(x)$, Jensen's inequality gives
\[
\tau\sum_{k=1}^N\|\beta^k(\cdot,0)\|_{L^2(\GR)}^2
\le \|b_0\|_{L^2(\Sigma_R)}^2.
\]
Taking $\varepsilon$ small, multiplying by $\tau$, and summing in $k$ yields
\begin{equation}\label{eq:discenergy}
\sup_{1\le k\le N}\|h^k\|_{L^2(\Om)}^2
+\tau\sum_{k=1}^N\|h^k\|_V^2
\le C,
\end{equation}
where $C$ is independent of $N$ and $\tau$. By the trace theorem,
\begin{equation}\label{eq:disctrace}
\tau\sum_{k=1}^N\|h^k\|_{L^2(\GR)}^2\le C.
\end{equation}

\subsection{Estimate of the discrete nonlinear time derivative}
From \eqref{eq:discrete}, for every $v\in V$,
\begin{align*}
\left|\left\langle\frac{B(\cdot,h^k)-B(\cdot,h^{k-1})}{\tau},v\right\rangle\right|
&\le T_1\|h^k\|_V\|v\|_V
+\|\beta^k(\cdot,h^k)\|_{L^2(\GR)}\|v\|_{L^2(\GR)}\\
&\quad+\|q^k\|_{V'}\|v\|_V
+C\|g^k\|_{H^{-1/2}(\GN)}\|v\|_V.
\end{align*}
By \eqref{eq:betagrowth}, \eqref{eq:discenergy}, \eqref{eq:disctrace}, and the trace theorem,
\begin{equation}\label{eq:disctime}
\tau\sum_{k=1}^N
\left\|\frac{B(\cdot,h^k)-B(\cdot,h^{k-1})}{\tau}\right\|_{V'}^2
\le C.
\end{equation}

\subsection{Interpolants and compactness}
Define the piecewise constant interpolants
\[
\bar h_\tau(t)=h^k,
\qquad
\bar z_\tau(t)=B(\cdot,h^k),
\qquad t\in I_k,
\]
and the piecewise affine interpolation
\[
\widehat z_\tau(t)
=B(\cdot,h^{k-1})
+\frac{t-t_{k-1}}{\tau}
\bigl(B(\cdot,h^k)-B(\cdot,h^{k-1})\bigr),
\qquad t\in I_k.
\]
Then
\[
\partial_t\widehat z_\tau
=\frac{B(\cdot,h^k)-B(\cdot,h^{k-1})}{\tau}
\quad\text{on }I_k,
\]
so \eqref{eq:disctime} gives
\begin{equation}\label{eq:ztime}
\|\partial_t\widehat z_\tau\|_{L^2(0,T;V')}\le C.
\end{equation}

We next use the spatial chain-rule assumption \eqref{eq:Bchain}. Since
\[
\nabla B(x,h^k)
=S(x,h^k)\nabla h^k+\nabla_xB(x,h^k),
\]
relations \eqref{eq:coeff}, \eqref{eq:Bx}, and \eqref{eq:discenergy} imply
\begin{equation}\label{eq:zH1}
\tau\sum_{k=1}^N\|B(\cdot,h^k)\|_{H^1(\Om)}^2\le C.
\end{equation}
Thus $\bar z_\tau$ is bounded in $L^2(0,T;H^1(\Om))$. Notice that no $H^1$ regularity is required of
$z^0=B(\cdot,h_0)$; only $z^0\in L^2(\Om)$ is used.

\begin{lemma}[Discrete compactness]\label{lem:disccompact}
Let $z^0\in L^2(\Om)$ and $z^k\in H^1(\Om)$ for $k\ge1$. Suppose
\[
\tau\sum_{k=1}^N\|z^k\|_{H^1(\Om)}^2\le C,
\qquad
\tau\sum_{k=1}^N\left\|\frac{z^k-z^{k-1}}{\tau}\right\|_{V'}^2\le C.
\]
Then the piecewise constant interpolation $\bar z_\tau$ is relatively compact in $L^2(0,T;L^2(\Om))$.
\end{lemma}

\begin{proof}
Set
\[
d^j:=\frac{z^j-z^{j-1}}{\tau}.
\]
For an integer $\ell\ge1$ and $s=\ell\tau$,
\[
z^{k+\ell}-z^k=\tau\sum_{j=k+1}^{k+\ell}d^j.
\]
By Cauchy--Schwarz,
\[
\|z^{k+\ell}-z^k\|_{V'}^2
\le \ell\tau^2\sum_{j=k+1}^{k+\ell}\|d^j\|_{V'}^2
=s\tau\sum_{j=k+1}^{k+\ell}\|d^j\|_{V'}^2.
\]
Multiplying by $\tau$ and summing over $k=0,\dots,N-\ell-1$, each index $j$
is counted at most $\ell$ times. Hence
\begin{align}
\int_0^{T-s}\|\bar z_\tau(t+s)-\bar z_\tau(t)\|_{V'}^2\,dt
&\le s\tau^2\ell\sum_{j=1}^N\|d^j\|_{V'}^2\notag\\
&=s^2\left(\tau\sum_{j=1}^N\|d^j\|_{V'}^2\right)
\le Cs^2.
\label{eq:Vprime-trans}
\end{align}
Thus the exponent in \eqref{eq:Vprime-trans} follows directly from the
$L^2$-bound on the discrete time derivative.

Now let $s>0$ be arbitrary and write
\[
s=\ell\tau+r,\qquad \ell\in\mathbb N_0,\quad 0\le r<\tau.
\]
By the triangle inequality,
\begin{align*}
\|\bar z_\tau(\cdot+s)-\bar z_\tau(\cdot)\|_{L^2(0,T-s;V')}
&\le
\|\bar z_\tau(\cdot+s)-\bar z_\tau(\cdot+\ell\tau)\|_{L^2(0,T-s;V')}\\
&\quad+
\|\bar z_\tau(\cdot+\ell\tau)-\bar z_\tau(\cdot)\|_{L^2(0,T-s;V')}.
\end{align*}
The second term is bounded by $C\ell\tau\le Cs$ by
\eqref{eq:Vprime-trans}. For the first term, since $0\le r<\tau$, the two
piecewise-constant functions differ only on subintervals of length $r$ adjacent
to the grid points, and on each such subinterval their difference equals one discrete jump. Therefore
\begin{align*}
\int_0^{T-s}
\|\bar z_\tau(t+s)-\bar z_\tau(t+\ell\tau)\|_{V'}^2\,dt
&\le r\sum_{j=1}^N\|z^j-z^{j-1}\|_{V'}^2\\
&=r\tau^2\sum_{j=1}^N\|d^j\|_{V'}^2\\
&\le Cr\tau\le C\tau^2.
\end{align*}
Consequently,
\begin{equation}\label{eq:Vprime-trans-general}
\int_0^{T-s}\|\bar z_\tau(t+s)-\bar z_\tau(t)\|_{V'}^2\,dt
\le C(s+\tau)^2.
\end{equation}

On the other hand, by the assumed $L^2(0,T;H^1(\Om))$ bound,
\[
\int_0^{T-s}
\|\bar z_\tau(t+s)-\bar z_\tau(t)\|_{H^1(\Om)}^2\,dt\le C.
\]
Since
\[
H^1(\Om)\Subset L^2(\Om)\hookrightarrow V',
\]
Ehrling's inequality gives, for every $\varepsilon>0$,
\[
\|w\|_{L^2(\Om)}^2
\le \varepsilon\|w\|_{H^1(\Om)}^2
+C_\varepsilon\|w\|_{V'}^2.
\]
Applying this inequality to
$w=\bar z_\tau(t+s)-\bar z_\tau(t)$, integrating in time, and using
\eqref{eq:Vprime-trans-general}, we find
\[
\limsup_{\tau\downarrow0}
\int_0^{T-s}\|\bar z_\tau(t+s)-\bar z_\tau(t)\|_{L^2(\Om)}^2\,dt
\le C\varepsilon+C_\varepsilon Cs^2.
\]
First let $s\downarrow0$ and then $\varepsilon\downarrow0$. It follows that
\[
\lim_{s\downarrow0}\,\limsup_{\tau\downarrow0}
\int_0^{T-s}\|\bar z_\tau(t+s)-\bar z_\tau(t)\|_{L^2(\Om)}^2\,dt=0.
\]
Together with the uniform $L^2(0,T;H^1(\Om))$ bound, the compactness criterion of
Simon \cite{SimonCompactness} yields relative compactness of $\bar z_\tau$ in
$L^2(0,T;L^2(\Om))$.
\end{proof}

Applying Lemma~\ref{lem:disccompact} with $z^k=B(\cdot,h^k)$ and using \eqref{eq:zH1} and \eqref{eq:disctime}, we obtain, after extraction,
\begin{equation}\label{eq:zstrong}
\bar z_\tau=B(\cdot,\bar h_\tau)\to z
\quad\text{strongly in }L^2(\QT).
\end{equation}
Moreover, from the definition of the affine interpolant and \eqref{eq:disctime},
\begin{equation}\label{eq:zhat-zbar}
\|\widehat z_\tau-\bar z_\tau\|_{L^2(0,T;V')}\le C\tau,
\end{equation}
so $\widehat z_\tau\to z$ strongly in $L^2(0,T;V')$.

Since $B(x,\cdot)$ is uniformly invertible and its inverse is $S_0^{-1}$-Lipschitz, there exists a measurable function $h$ defined by
\[
h(x,t)=B(x,\cdot)^{-1}(z(x,t)),
\]
and \eqref{eq:zstrong} implies
\begin{equation}\label{eq:hstrong}
\bar h_\tau\to h
\quad\text{strongly in }L^2(\QT).
\end{equation}
Moreover, \eqref{eq:discenergy} yields, after extraction,
\begin{align}
\bar h_\tau&\rightharpoonup h &&\text{in }L^2(0,T;V),\label{eq:hweakV}\\
\bar h_\tau&\stackrel{*}{\rightharpoonup}h &&\text{in }L^\infty(0,T;L^2(\Om)).\label{eq:hweakstar}
\end{align}
By \eqref{eq:zH1}, $\bar z_\tau$ is also weakly compact in $L^2(0,T;H^1(\Om))$; hence, using \eqref{eq:zstrong},
\[
B(\cdot,h)=z\in L^2(0,T;H^1(\Om)).
\]
Finally, by \eqref{eq:ztime}, after extraction there exists $\xi\in L^2(0,T;V')$ such that
\[
\partial_t\widehat z_\tau\rightharpoonup\xi
\quad\text{in }L^2(0,T;V').
\]
On the other hand, \eqref{eq:zhat-zbar} and \eqref{eq:zstrong} imply
$\widehat z_\tau\to B(\cdot,h)$ strongly in $L^2(0,T;V')$. Hence, for every
$v\in V$ and $\varphi\in C_c^\infty(0,T)$,
\begin{align*}
\int_0^T\langle\xi,v\rangle\varphi(t)\,dt
&=\lim_{\tau\to0}\int_0^T\langle\partial_t\widehat z_\tau,v\rangle\varphi(t)\,dt\\
&=-\lim_{\tau\to0}\int_0^T(\widehat z_\tau,v)_{L^2}\varphi'(t)\,dt\\
&=-\int_0^T(B(\cdot,h),v)_{L^2}\varphi'(t)\,dt.
\end{align*}
Therefore $\xi=\partial_tB(\cdot,h)$ in the sense of distributions, and consequently
\begin{equation}\label{eq:zdtweak}
\partial_t\widehat z_\tau\rightharpoonup \partial_tB(\cdot,h)
\quad\text{in }L^2(0,T;V').
\end{equation}

\subsection{Strong compactness of traces}
For bounded Lipschitz domains,
\[
\|v\|_{L^2(\partial\Om)}^2
\le C\|v\|_{L^2(\Om)}\|v\|_{H^1(\Om)}.
\]
Apply this to $v=\bar h_\tau-h$, integrate in time, and use \eqref{eq:hstrong} together with the uniform $L^2(0,T;H^1)$ bound. Then
\begin{equation}\label{eq:tracecompact}
\bar h_\tau\to h
\quad\text{strongly in }L^2(0,T;L^2(\partial\Om)),
\end{equation}
and in particular on $\Sigma_R$.

\subsection{Passage to the limit}
Let $\bar q_\tau,\bar g_\tau$ denote the piecewise constant time averages of the data, and let
\[
\bar\beta_\tau(x,t,r)=\beta^k(x,r)\qquad (t\in I_k).
\]
The discrete equation may be written for a.e. $t\in(0,T)$ as
\begin{align}
\langle\partial_t\widehat z_\tau,v\rangle
&+\int_\Om T(x,\bar h_\tau)\nabla\bar h_\tau\cdot\nabla v\,dx
+\int_{\GR}\bar\beta_\tau(x,t,\bar h_\tau)v\,dS\notag\\
&=\langle\bar q_\tau,v\rangle-\langle\bar g_\tau,\Tr v\rangle.
\label{eq:interpdisc}
\end{align}

Because $\bar h_\tau\to h$ strongly in $L^2(\QT)$, after a subsequence $\bar h_\tau\to h$ almost everywhere. Continuity and boundedness of $T$ imply
\[
T(x,\bar h_\tau)\to T(x,h)
\quad\text{a.e. in }\QT.
\]
For each fixed $v\in V$,
\[
[T(x,\bar h_\tau)-T(x,h)]\nabla v\to0
\quad\text{strongly in }L^2(\QT)
\]
by dominated convergence. Together with \eqref{eq:hweakV}, this gives
\begin{equation}\label{eq:diffusionlimit}
\int_0^T\!\int_\Om T(x,\bar h_\tau)\nabla\bar h_\tau\cdot\nabla v\,\psi
\to
\int_0^T\!\int_\Om T(x,h)\nabla h\cdot\nabla v\,\psi
\end{equation}
for every bounded time test function $\psi$.

For the boundary term, define the time-averaging operator $P_\tau$ on $L^2(\Sigma_R)$ by
\[
(P_\tau F)(x,t)=\frac1\tau\int_{I_k}F(x,s)\,ds,
\qquad t\in I_k.
\]
Since $\bar h_\tau$ is constant on each $I_k$,
\begin{equation}\label{eq:beta-average-id}
\bar\beta_\tau(x,t,\bar h_\tau(x,t))
=P_\tau\bigl[\beta(\cdot,\cdot,\bar h_\tau)\bigr](x,t).
\end{equation}
By \eqref{eq:tracecompact}, $\bar h_\tau\to h$ strongly in $L^2(\Sigma_R)$. The Carath\'eodory property and the linear growth assumption imply continuity of the associated Nemytskii operator on $L^2(\Sigma_R)$; hence
\[
\beta(x,t,\bar h_\tau)\to\beta(x,t,h)
\quad\text{strongly in }L^2(\Sigma_R).
\]
Because $P_\tau$ is an $L^2$ contraction and $P_\tau F\to F$ strongly in $L^2$ for every fixed $F\in L^2(\Sigma_R)$,
\begin{align*}
\|\bar\beta_\tau(\bar h_\tau)-\beta(h)\|_{L^2(\Sigma_R)}
&\le \|P_\tau[\beta(\bar h_\tau)-\beta(h)]\|_{L^2}
+\|P_\tau[\beta(h)]-\beta(h)\|_{L^2}\\
&\longrightarrow0.
\end{align*}
Therefore
\begin{equation}\label{eq:robinlimit}
\int_0^T\!\int_{\GR}\bar\beta_\tau(x,t,\bar h_\tau)v\psi\,dSdt
\to
\int_0^T\!\int_{\GR}\beta(x,t,h)v\psi\,dSdt.
\end{equation}
The piecewise averages of $q$ and $g$ converge strongly in their respective $L^2$ spaces.

Finally, choose $v\in V$ and $\psi\in C^1([0,T])$ with $\psi(T)=0$, multiply \eqref{eq:interpdisc} by $\psi$, and integrate in time. Integration by parts gives
\begin{align*}
-\int_0^T(\widehat z_\tau,v)\psi'\,dt
&+\int_0^T\!\int_\Om T(x,\bar h_\tau)\nabla\bar h_\tau\cdot\nabla v\,\psi\,dxdt\\
&+\int_0^T\!\int_{\GR}\bar\beta_\tau(x,t,\bar h_\tau)v\psi\,dSdt\\
&=\int_0^T\bigl(\langle\bar q_\tau,v\rangle-\langle\bar g_\tau,\Tr v\rangle\bigr)\psi\,dt
+(B(\cdot,h_0),v)\psi(0).
\end{align*}
Using \eqref{eq:zstrong}, \eqref{eq:zhat-zbar}, \eqref{eq:diffusionlimit}, and \eqref{eq:robinlimit}, we pass to the limit and obtain precisely \eqref{eq:weakint}. This proves Theorem~\ref{thm:existence}.
\qed

\section{Return to the physical Moche coefficients}\label{sec:physical}
The preceding theorem concerns the uniformly positive extensions of the physical constitutive laws. Its relation with the original aquifer problem is immediate on the admissible saturated range.

\begin{corollary}[Consistency with the physical model]\label{cor:physical}
Let $h$ be a weak solution furnished by Theorem~\ref{thm:existence}. Suppose in addition that
\[
h(x,t)\in I=[m,M]
\quad\text{for a.e. }(x,t)\in\QT.
\]
Assume also that, on $I$, the abstract boundary law coincides with the physical one,
$\beta(x,t,r)=\alpha(r)(r-h_{\rm ext}(x,t))$. Assume, in addition, that the abstract source term is the physical one, $q=N$ (viewed as an element of $L^2(0,T;V')$). Then $S(x,h)=S_{\rm ph}(x,h)$ and $T(x,h)=T_{\rm ph}(x,h)$ almost everywhere, and $h$ is a weak solution of the physical Moche model \eqref{eq:physical2} with the same mixed boundary conditions.
\end{corollary}

\begin{remark}
The corollary is a consistency statement, not an invariant-region theorem. Establishing that a given set of hydrogeological data forces the solution to remain in a prescribed interval $I$ requires additional comparison or barrier assumptions. Separating these two issues prevents the existence theorem from hiding a circular saturated-regime hypothesis.
\end{remark}

For the Robin law in the original formulation,
\[
\beta(x,t,r)=\alpha(r)(r-h_{\rm ext}(x,t)).
\]
Theorem~\ref{thm:existence} applies whenever this map is nondecreasing in $r$ and satisfies the linear-growth assumption. A particularly transparent case is $\alpha(r)\equiv\alpha_*>0$, for which
\[
\beta(x,t,r)=\alpha_*r-\alpha_*h_{\rm ext}(x,t)
\]
is strongly monotone and \eqref{eq:betagrowth} follows from $h_{\rm ext}\in L^2(\Sigma_R)$.

\section{The degenerate saturated-thickness limit}\label{sec:degenerate}
For the physical transmissivity
\[
T_{\rm ph}(x,h)=K(h-\eta(x)),
\]
uniform ellipticity is lost when $h=\eta$. Introducing
\[
u=h-\eta,
\]
we obtain, whenever $u\ge0$,
\begin{equation}\label{eq:degenerate}
\partial_t\left(n_eu+\frac{S_s}{2}u^2\right)
-\frac K2\Delta(u^2)
-K\operatorname{div}(u\nabla\eta)=N.
\end{equation}
Thus complete local depletion leads to a porous-medium-type diffusion with an additional topographic drift. This is precisely the type of degeneracy for which monotonicity and elliptic--parabolic methods become natural; see \cite{AltLuckhaus,Showalter}.

A regularization may be obtained by replacing $u$ in the mobility by a smooth positive approximation $\theta_\varepsilon(u)\ge\varepsilon$. However, if the source term contains sufficiently strong extraction, preserving the physical constraint $u\ge0$ after complete local depletion is not merely a compactness issue: a compatibility or complementarity condition is needed. The fully depleted problem is therefore closer to a degenerate filtration or free-boundary model and is left for future work.

\section{Concluding remarks}
The hydrogeological geometry of the Moche--CHAVIMOCHIC aquifer leads naturally to a nonlinear two-dimensional parabolic problem with three distinct boundary mechanisms. The analysis above provides a rigorous weak-solvability result for a uniformly saturated extension of the physical model and identifies the precise consistency condition under which that solution solves the original aquifer equations. The implicit Euler--Rothe scheme avoids the unjustified time-derivative estimate that arises in a direct Galerkin treatment and yields compactness through the storage potential $B(x,h)$. Two natural extensions remain: an invariant-region result ensuring preservation of a prescribed saturated interval, and a genuinely degenerate theory allowing the saturated thickness to vanish.

\end{document}